\documentclass{article}
\usepackage[utf8]{inputenc}
\usepackage{amsmath}
\usepackage{amsfonts}
\usepackage{amssymb}
\usepackage{tikz-cd}
\usepackage{color,soul}
\usepackage{array}
\usepackage{zref-savepos}
\usepackage{amsthm}
\usepackage{multirow, bigstrut}
\usepackage[font={small,it}]{caption}
\usepackage{tikz}
\usepackage{amsmath}
\usetikzlibrary{arrows}

\usepackage{hyperref}
\hypersetup{
    colorlinks=true,
    linkcolor=blue,
    filecolor=magenta,  
    urlcolor=cyan,
}

\usetikzlibrary{calc}
\usepackage{comment}
\usepackage{enumerate}
\usepackage{tikz}
\usetikzlibrary{shapes.geometric}

\newtheorem{theorem}{Theorem}[section]
\newtheorem*{theorem*}{Theorem}
\newtheorem{lemma}[theorem]{Lemma}
\newtheorem{proposition}[theorem]{Proposition}
\newtheorem{corollary}[theorem]{Corollary}
\newtheorem{conjecture}[theorem]{Conjecture}

\newtheorem{fact}[theorem]{Fact}

\theoremstyle{definition}
\newtheorem{definition}[theorem]{Definition}

\theoremstyle{plain}

\newcommand{\C}{\mathbb{C}}
\newcommand{\R}{\mathbb{R}}
\newcommand{\N}{\mathbb{N}}
\newcommand{\Z}{\mathbb{Z}}
\newcommand{\K}{\mathbb{K}}

\DeclareMathOperator{\Hmg}{Hmg}
\DeclareMathOperator{\Pol}{Pol}

\DeclareMathOperator{\Symb}{Symb}

\DeclareMathOperator{\cpc}{Cap}
\DeclareMathOperator{\Newt}{Newt}
\DeclareMathOperator{\supp}{supp}

\title{Counting Matchings via Capacity Preserving Operators}
\author{Leonid Gurvits and Jonathan Leake}

\begin{document}

\sloppy

\maketitle

\begin{abstract}
    The notion of the capacity of a polynomial was introduced by Gurvits around 2005, originally to give drastically simplified proofs of the Van der Waerden lower bound for permanents of doubly stochastic matrices and Schrijver's inequality for perfect matchings of regular bipartite graphs. Since this seminal work, the notion of capacity has been utilized to bound various combinatorial quantities and to give polynomial-time algorithms to approximate such quantities (e.g., the number of bases of a matroid). These types of results are often proven by giving bounds on how much a particular differential operator can change the capacity of a given polynomial. In this paper, we unify the theory surrounding such capacity preserving operators by giving tight capacity preservation bounds for all nondegenerate real stability preservers. We then use this theory to give a new proof of a recent result of Csikv\'ari, which settled Friedland's lower matching conjecture.
\end{abstract}

\section{Introduction}


Over the past few decades, the theory of real stable polynomials has found various applications, particularly within combinatorics, probability, computer science, and optimization (e.g., see \cite{bb2}, \cite{choe2004homogeneous} and references therein). Classic examples include the multivariate matching polynomial and the spanning tree polynomial, both of which are real stable for any given graph. The role that polynomials often play in these applications is that of conceptual unification: various natural operations that one may apply to a given type of object can often be represented as natural operations applied to associated polynomials. For the matching polynomial deletion and contraction correspond to certain derivatives, and for the spanning tree polynomial this idea extends to the minors of a matroid in general. Even in optimization (specifically hyperbolic programming), certain relaxations of convex domains translate into directional derivatives of associated polynomials in a similar way \cite{renegar2006hyperbolic}.

Real stability then adds extra information that may be useful to track. For example, the real stability of the matching polynomial easily implies that the number of size $k$ matchings of a graph forms a log-concave sequence \cite{HL}. As it turns out, real stability is far more generally connected to log-concavity than this, and we will see this at play in the main results of this paper. Specifically, the so-called strong Rayleigh inequalities (see \cite{strongrayleigh}) will play a crucial role in our analysis. Related inequalities have recently have gained importance through the exciting work on a so-called Hodge theory for matroids \cite{adiprasito2018hodge}. Results similar to those discussed here can even be extended to basis generating polynomials of matroids in general (not all of which are real stable); see \cite{huh2018correlation} and \cite{anari2018log}.

The particular line into which this paper falls then begins with the work of the first author, who in a series of papers (e.g., see \cite{gurvits2008van}) gave a vast generalization of the Van der Waerden lower bound for permanents of doubly stochastic matrices and the Schrijver lower bound on the number of perfect matchings of regular graphs. In particular, he showed that a related inequality holds for real stable polynomials in general, and then derives each of the referenced results as corollaries. His inequality describes how much the derivative can affect a particular analytic quantity called the capacity of a polynomial, and our main goal in this paper is to extend this bound to a much larger class of linear operators on polynomials.

\subsection{Prior Results}

Before stating our results, we recall a few results regarding capacity and stable polynomials which will be crucial to the framing of our results. First, we give the definition of polynomial capacity, where here (and throughout this paper) we use the notation $x^\alpha := \prod_k x_k^{\alpha_k}$.

\begin{definition}
    Given a polynomial $p \in \R[x_1,...,x_n]$ with non-negative coefficients and a vector $\alpha \in \R^n$ with non-negative entries, we define the \emph{$\alpha$-capacity} of $p$ as:
    \[
        \cpc_\alpha(p) := \inf_{x > 0} \frac{p(x)}{x^\alpha} = \inf_{x_1,\ldots,x_n > 0} \frac{p(x_1,\ldots,x_n)}{x_1^{\alpha_1} \cdots x_n^{\alpha_n}}.
    \]
\end{definition}

This definition was motivated by Gurvits' capacity inequality for the derivative, which we state now.

\begin{theorem}[\cite{gurvits2006hyperbolic}] \label{thm:gurvits_intro}
    Let $p \in \R[x_1,...,x_n]$ be a real stable polynomial of degree at most $\lambda_k$ in $x_k$ with non-negative coefficients. Then:
    \[
        \frac{\cpc_{(1^{n-1})}\left(\left.\partial_{x_k}p\right|_{x_k=0}\right)}{\cpc_{(1^n)}(p)} \geq \left(\frac{\lambda_k - 1}{\lambda_k}\right)^{\lambda_k - 1}.
    \]
    Here, $(1^j)$ denotes the all-ones vector of length $j$.
\end{theorem}

The way one should interpret this result is as a statement about the capacity preservation properties of the derivative. That is, taking a partial derivative of a real stable polynomial (and then evaluating to 0) can only decrease the capacity of that polynomial by at most the stated multiplicative factor.

For those familiar with the real stability literature, the concept of preservation properties of a linear operator (specifically that of the derivative here) is not new. Perhaps the most essential result in the theory is the Borcea-Br\"and\'en characterization \cite{bb1}, which characterizes all linear operators on polynomials which preserve the property of being real stable. This result relies on the concept of the symbol of a linear operator $T$, denoted $\Symb(T)$, which is a single specific polynomial (or power series) associated to $T$. We give the gist of the characterization here but delay the definition of the symbol and the formal statement of the theorem until later.

\begin{theorem}[Borcea-Br\"and\'en characterization \cite{bb1}; see Theorems \ref{bbbounded} and \ref{bbtrans}]
    Let $T$ be a real linear operator on polynomials. Then morally speaking, $T$ preserves the property of being real stable if and only if $\Symb(T)$ is real stable.
\end{theorem}

\subsection{Our Results}

The Borcea-Br\"and\'en characterization says that the symbol of a linear operator $T$ holds the real stability preservation information of $T$. In this paper, we make use of this concept by showing that the symbol also holds the capacity preservation information of $T$. That is, we combine the ideas of Gurvits and of Borcea and Br\"and\'en to create a theory of capacity preserving operators. Our main results in this direction are stated slightly informally in the following theorems. Note that $\Symb$ will take on two different definitions in the formal statements of these results (see Definitions \ref{symbbounded} and \ref{symbtrans}), and we will explicate this rigorously later.

\begin{theorem}[= Theorem \ref{mainthmbounded}; Bounded degree case]
    Let $T$ be a linear operator on polynomials of degree at most $\lambda_k$ in $x_k$, such that $\Symb(T)$ is real stable with non-negative coefficients. Further, let $p \in \R[x_1,...,x_n]$ be a real stable polynomial of degree at most $\lambda_k$ in $x_k$ with non-negative coefficients. Then for any sensible non-negative vectors $\alpha,\beta \in \R^n$:
    \[
        \frac{\cpc_\beta(T(p))}{\cpc_\alpha(p)} \geq \frac{\alpha^\alpha(\lambda-\alpha)^{\lambda-\alpha}}{\lambda^\lambda} \cpc_{(\alpha,\beta)}(\Symb(T)).
    \]
    Further, this bound is tight for fixed $T,\alpha,\beta$.
\end{theorem}

\begin{theorem}[= Theorem \ref{mainthmtrans}; Unbounded degree case]
    Let $T$ be a linear operator on polynomials of any degree, such that $\Symb(T)$ is in the Laguerre-P\'olya class\footnote{For those unfamiliar, the Laguerre-P\'olya class consists of limits of real stable polynomials.} with non-negative coefficients. Further, let $p \in \R[x_1,...,x_n]$ be any real stable polynomial with non-negative coefficients. Then for any sensible non-negative vectors $\alpha,\beta \in \R^n$:
    \[
        \frac{\cpc_\beta(T(p))}{\cpc_\alpha(p)} \geq e^{-\alpha} \alpha^\alpha \cpc_{(\alpha,\beta)}(\Symb(T)).
    \]
    Further, this bound is tight for fixed $T,\alpha,\beta$.
\end{theorem}

Using these theorems, we are able to reprove several results. The first of these is Gurvits' theorem, 
which he used to obtain the corollaries mentioned above: the Van der Waerden lower bound (see \cite{falikman1981proof} and \cite{erorychev1981proof} for the original resolution of this conjecture) and Schrijver's inequality \cite{schrijver1998counting}. We reprove Gurvits' theorem using the capacity preservation theory, which amounts to a very basic computation for $T = \left.\partial_{x_k}\right|_{x_k=0}$.

Our main application is then related to counting matchings of regular bipartite graphs. Counting the number of matchings in a graph is related to the monomer-dimer problem of evaluating/approximating the monomer-dimer partition function of a given graph. This problem is one of the oldest and most important problems in statistical physics, with much of the importance being due to the famous paper of Heilmann and Lieb \cite{HL}. Their results on the location of phase transitions of the partition function (i.e., the location of zeros of the matching polynomial) have had widespread influence, even playing a crucial role in the (somewhat) recent resolution of the Kadison-Singer conjecture by Marcus, Spielman, and Srivastava \cite{marcus2015interlacing}.

Specifically, in Section \ref{sec:biregular_graphs} we give a simpler proof of Csikv\'ari's bound on the number of $k$-matchings of a biregular bipartite graph \cite{csikvari2014lower}. This result generalizes Schrijver's inequality and is actually a strengthening of Friedland's lower matching conjecture (see \cite{friedland2008number}). The computations involved in this new proof never exceed the level of basic calculus. This was one of the most remarkable features of Gurvits' original result, and this theme continues to play out here. We state Csikv\'ari's result now.

\begin{theorem}[\cite{csikvari2014lower}]
    Let $G$ be an $(a,b)$-biregular bipartite graph with $(m,n)$-bipartitioned vertices (so that $am = bn$ is the number of edges of $G$). Then the number of size-$k$ matchings of $G$ is bounded as follows:
    \[
        \mu_k(G) \geq \binom{n}{k}(ab)^k \frac{m^m(ma-k)^{ma-k}}{(ma)^{ma}(m-k)^{m-k}}.
    \]
\end{theorem}

We also note that partial results toward such a bound, using techniques similar to those used in this paper, were achieved prior to Csikv\'ari's result. First in \cite{friedland2008number}, the original lower matching conjecture (for regular graphs) was settled for degree $2$ and for $k \leq 4$. Further in \cite{friedland2008lower}, partial results are given for the asymptotic version of the lower matching conjecture, and in \cite{gurvits2011unleashing} this asymptotic version is settled. In these last two papers, stable polynomials and results derived from Theorem \ref{thm:gurvits_intro} were used.

Beyond these specific applications, one of the main purposes of this paper is to unify the various results that fit into the lineage of the concept of capacity. This includes inequalities for the permanent, the mixed discriminant, and the number of perfect matchings of bipartite graphs \cite{gurvits2008van, gurvits2011unleashing}, inequalities for coefficients of stable and log-concave polynomials \cite{gurvits2009multivariate}, approximation algorithms counting bases of stable matroids \cite{anari2017generalization, straszak2017real}, approximation algorithms for counting the intersection of two general matroids \cite{anari2018log}, capacity preservation results for elementary symmetric differential operators \cite{zackrisson2017coefficients}, and capacity preservation results for differential operators on multiaffine and general degree polynomials \cite{anari2017generalization}.

The rest of this paper is outlined as follows. In \S\ref{preliminaries}, we discuss some preliminary facts about real stability and capacity. In \S\ref{applications}, we discuss applications of the capacity preservation theory. In \S\ref{mainineq}, we prove the main inequalities. In \S\ref{continuity}, we discuss some continuity properties of capacity.

\section{Preliminaries} \label{preliminaries}

We first discuss some basics of the theories of real stability and of capacity. This section will consist mainly of well-known and/or standard results that will enable us to state our main results in the next section formally. Other results needed to prove the main theorems will be left to later sections.

\subsection{Notation}

Let $\C,\R,\Z,\N$ denote the complex numbers, real numbers, integers, and positive integers respectively. Also, let $\R_+$ and $\R_{++}$ denote the non-negative and positive reals respectively, and let $\Z_+$ denote the non-negative integers. With this we let $\K[x_1,...,x_n]$ denote the set of polynomials with coefficients in $\K$, where $\K$ can be any of the previously defined sets of numbers. Further, for $\lambda \in \Z_+^n$ we let $\K^\lambda[x_1,...,x_n]$ denote the set of polynomials of degree at most $\lambda_k$ in $x_k$ with coefficients in $\K$.

For $\mu,\lambda \in \Z_+^n$, we define $\mu! := \prod_k (\mu_k!)$ and $\binom{\lambda}{\mu} := \frac{\lambda!}{\mu!(\lambda-\mu)!}$. For $x,\alpha \in \R_+^n$ we define $\alpha \leq x$ via $\alpha_k \leq x_k$ for all $k$, we define $x\alpha := \prod_k x_k \alpha_k$, and we define $x^\alpha := \prod_k x_k^{\alpha_k}$ as in the definition of capacity. We also let $(1^n) \in \Z_+^n$ denote the all-ones vector of length $n$. Finally for $p \in \K[x_1,...,x_n]$ we let $p_\mu$ denote the coefficient of $p$ corresponding to the term $x^\mu$, and in this vein we will sometimes let $x$ refer to the vector of variables $(x_1,...,x_n)$.

\subsection{Real Stability}

We call a polynomial $p \in \R[x_1,...,x_n]$ real stable if $p \equiv 0$ or $p(x_1,...,x_n) \neq 0$ whenever all $x_k$ are in the upper half-plane. Note that for $n=1$, $p$ is a univariate polynomial and real stability is equivalent to having all real roots. The theory of stable polynomials enjoys a nice inductive structure deriving from a large class of linear operators on polynomials which preserve the property of being real stable. We called such operators \emph{real stability preservers}, and the most basic of these are given as follows.

\begin{proposition}[Basic real stability preservers]\label{rsclosure}
    Let $p \in \R^\lambda[x_1,...,x_n]$ be real stable. Then the following are also real stable.
    \begin{enumerate}
        \item Permutation: $p(x_{\sigma(1)}, ..., x_{\sigma(n)})$ for any $\sigma \in S_n$
        \item Scaling: $p(a_1x_1,...,a_nx_n)$ for any fixed $a \in \R_+^n$
        \item Specialization: $p(b,x_2,...,x_n)$ for any fixed $b \in \R$
        \item Inversion: $x^\lambda p(x_1^{-1},...,x_n^{-1})$
        \item Differentiation: $\partial_{x_k}p$ for any $k$
        \item Diagonalization: $\left.p\right|_{x_j=x_k}$ for any $j,k$
    \end{enumerate}
\end{proposition}

A classical but more interesting real stability preserver is \emph{polarization}. Polarization plays a crucial role in the theory of real stability preservers, as it allows one to restrict to polynomials of degree at most 1 in every variable. We will see later that polarization also plays a crucial role in the theory of capacity preservers.

\begin{definition}
    Given $q \in \R^d[x]$, we define $\Pol^d(q)$ to be the unique symmetric $f \in \R^{(1^d)}[x_1,...,x_d]$ such that $f(x,...,x) = q(x)$. Given $p \in \R^\lambda[x_1,...,x_n]$, we define $\Pol^\lambda(p) := (\Pol^{\lambda_1} \circ \cdots \circ \Pol^{\lambda_n})(p)$, where $\Pol^{\lambda_k}$ acts on the variable $x_k$ for each $k$. Note that $\Pol^\lambda(p) \in \R^{(1^{\lambda_1 + \cdots + \lambda_n})}[x_{1,1},...,x_{n,\lambda_n}]$.
\end{definition}

\begin{proposition}[\cite{walsh}] \label{rspol}
    Given $p \in \R^\lambda[x_1,...,x_n]$, we have that $p$ is real-stable iff $\Pol^\lambda(p)$ is real stable.
\end{proposition}

Beyond these basic real stability preservers, various preservation results regarding different classes of operators have been proven over the past century or so. In 2008 many of these results were encapsulated and vastly generalized in the Borcea-Br\"and\'en characterization, which gives a useful equivalent condition for a linear operator to be a real stability preserver. We mentioned this result in the introduction, and now we present it formally. To that end, we first define the \emph{symbol of an operator}, a crucial concept to the rest of this paper.

\begin{definition}[Bounded-degree symbol] \label{symbbounded}
    Given a linear operator
    \[
        T: \R^\lambda[x_1,...,x_n] \to \R^\gamma[x_1,...,x_m],
    \]
    we define $\Symb^\lambda(T) \in \R^{(\lambda,\gamma)}[z_1,...,z_n,x_1,...,x_m]$ as follows, where $T$ acts only on the $x$ variables:
    \[
        \Symb^\lambda(T) := T[(1+xz)^\lambda] = \sum_{0 \leq \mu \leq \lambda} \binom{\lambda}{\mu} z^\mu T(x^\mu).
    \]
    We may simply write $\Symb(T)$ when $\lambda$ is clear from the context. (Note that this definition is slightly different from that of \cite{bb1}, but this difference is inconsequential.)
\end{definition}

The characterization then essentially says that $T$ preserves real stability if and only if $\Symb^\lambda(T)$ is real stable, with the exception of a certain degeneracy case. We state the full statement of the characterization, for bounded-degree operators.

\begin{theorem}[Borcea-Br\"and\'en] \label{bbbounded}
    Let $T: \R^\lambda[x_1,...,x_n] \to \R^\gamma[x_1,...,x_m]$ be a linear operator on polynomials. Then $T$ preserves real stability if and only if one of the following holds.
    \begin{enumerate}
        \item $\Symb^\lambda(T)$ is real stable.
        \item $\Symb^\lambda(T)(z_1,...,z_n,-x_1,...,-x_m)$ is real stable.
        \item The image of $T$ is of degree at most 2 and consists only of real stable polynomials.
    \end{enumerate}
\end{theorem}

Notice that the above definition and result deal only with operators which only allow inputs up to a certain fixed degree. And this is important to note, as the symbol changes based upon which degree is being considered. For operators which do not inherently depend on some fixed maximum degree (e.g., the derivative), there is another symbol definition and characterization result.

Of course, the degree of the symbol above is the same as the maximum degree of the input polynomials. So if one were to define a symbol for operators with no bound on the input degree, it is likely that the symbol would not have a bound on its degree. This is where the \emph{Laguerre-P\'olya class} comes in. This is a class of entire functions in $\C^n$, defined as follows.

\begin{definition}
    A function $f$ is said to be in the $\mathcal{LP}$ (Laguerre-P\'olya) class in the variables $x_1,...,x_n$, if $f$ is the limit (uniformly on compact sets) of real stable polynomials in $\R[x_1,...,x_n]$. If $f$ is the limit of real stable polynomials in $\R_+[x_1,...,x_n]$, then we say $f$ is in the $\mathcal{LP}_+$ class. In these cases, we write $f \in \mathcal{LP}[x_1,...,x_n]$ and $f \in \mathcal{LP}_+[x_1,...,x_n]$ respectively.
\end{definition}

There are interesting equivalent definitions for these classes of functions (e.g., see \cite{craven1989jensen}), but we omit them here. With this class of functions we can state the Borcea-Br\"and\'en characterization for operators with no dependence on the degree of the input polynomial. First though we need to define the ``transcendental'' symbol.

\begin{definition}[Transcendental symbol] \label{symbtrans}
    Given a linear operator
    \[
        T: \R[x_1,...,x_n] \to \R[x_1,...,x_m],
    \]
    we define $\Symb^\infty(T)$ as a formal power series in $z_1,....,z_n$ (with polynomial coefficients in $x_1,...,x_m$) as follows, where $T$ acts only on the $x$ variables:
    \[
        \Symb^\infty(T) := T[e^{x \cdot z}] = \sum_{0 \leq \mu} \frac{1}{\mu!} z^\mu T(x^\mu).
    \]
\end{definition}

\begin{theorem}[Borcea-Br\"and\'en] \label{bbtrans}
    Let $T: \R[x_1,...,x_n] \to \R[x_1,...,x_m]$ be a linear operator on polynomials. Then $T$ preserves real stability if and only if one of the following holds.
    \begin{enumerate}
        \item $\Symb^\infty(T) \in \mathcal{LP}[z_1,....,z_n,x_1,...,x_m]$
        \item $\Symb^\infty(T)(z_1,...,z_n,-x_1,...,-x_m) \in \mathcal{LP}[z_1,....,z_n,x_1,...,x_m]$
        \item The image of $T$ is of degree at most 2 and consists only of real stable polynomials.
    \end{enumerate}
\end{theorem}

\subsection{Capacity} \label{capss}

Recall the definition of capacity:
\[
    \cpc_\alpha(p) := \inf_{x > 0} \frac{p(x)}{x^\alpha}
\]
In general, the  conceptual meaning of capacity is not completely understood. However, in this section we hope to illuminate some of its basic features. This will include its connections to the coefficients of a polynomial, to probabilistic interpretations of polynomials, to the AM-GM inequality, and to the Legendre (Fenchel) transformation.

As discussed in the introduction, the sort of capacity results we will be interested in are those of capacity preservation (that is, bounds on how much the capacity can change under various operations). In fact, our use of the Borcea-Br\"and\'en characterization consists in combining it with capacity bounds in order to give something like a characterization of capacity preservers. This can be seen as an analytic refinement of the characterization: not only do such operators preserve stability, but they also preserve capacity. That said, we now state a few basic properties and interpretations of capacity that will be needed to state and discuss this analytic refinement. First recall the definitions of the Newton polytope and the support of a polynomial.

\begin{definition}
    Given $p \in \R[x_1,...,x_n]$, the \emph{Newton polytope} of $p$, denoted $\Newt(p)$, is the convex hull of the support of $p$. The \emph{support} of $p$, denoted $\supp(p)$, is the set of all $\mu \in \Z_+^n$ such that $x^\mu$ has a non-zero coefficient in $p$.
\end{definition}

Capacity is perhaps most basically understood as a quantity which mediates between the coefficients of $p$ and the evaluations of $p$. For example, if $\mu \in \supp(p)$ then:
\[
    p_\mu \leq \cpc_\mu(p) \leq p(1,...,1).
\]
Capacity can also be understood probabilistically. If $p \in \R_+^{(1^n)}[x_1,...,x_n]$ and $p(1,...,1) = 1$, then $p$ can be considered as the probability generating function for some discrete distribution on $\supp(p)$. In this case, a simple proof demonstrates:

\begin{fact} \label{fact:marginals}
    Let $p \in \R_+^{(1^n)}[x_1,...,x_n]$ be the probability generating function for some distribution $\nu$. Then:
    \begin{enumerate}
        \item $0 \leq \cpc_\alpha(p) \leq 1$ for all $\alpha \in \R_+^n$.
        \item $\cpc_\alpha(p) = 1$ if and only if $\alpha$ is the vector of marginal probabilities of $\nu$.
    \end{enumerate}
\end{fact}
\begin{proof}
    $(1)$ is straightforward, and $(2)$ follows from concavity of $\log$ (e.g., see \cite{gurvits2008van}, Fact 2.2) and the fact that $\cpc_\alpha(p) = 1$ implies $\frac{p(x)}{x^\alpha}$ is minimized at the all-ones vector.
\end{proof}

The following ``log-exponential polynomial'' associated to $p$ has some nice properties which often makes it convenient to use in the context of capacity. These properties also shed light on the potential connection between capacity, convexity, and the Legendre transformation (consider the expressions which show up in Fact \ref{convexcap} below).


\begin{definition}
    Given a polynomial $p \in \R_+[x_1,...,x_n]$, we let capitalized $P$ denote the following function:
    \[
        P(x) := \log(p(\exp(x))) = \log \sum_\mu p_\mu e^{\mu \cdot x}.
    \]
\end{definition}

\begin{fact}\label{convexcap}
    Given $p \in \R_+[x_1,...,x_n]$, consider $P$ as defined above. We have:
    \begin{enumerate}
        \item $\cpc_\alpha(p) = \exp \inf_{x \in \R^n} (P(x) - \alpha \cdot x)$
        \item $P(x) - \alpha \cdot x$ is convex in $\R^n$ for any $\alpha \in \R^n$.
    \end{enumerate}
\end{fact}

The next result is essentially a corollary of the AM-GM inequality. In a sense, this inequality is the foundational result that makes the notion of capacity so useful. Because of this we provide a partial proof of the following result, taken from \cite{anari2017generalization}.

\begin{fact}\label{newtoncap}
    For $p \in \R_+[x_1,...,x_n]$, $P$ defined as above, and $\alpha \in \R^n_+$, the following are equivalent.
    \begin{enumerate}
        \item $\alpha \in \Newt(p)$
        \item $\cpc_\alpha(p) > 0$
        \item $P(x) - \alpha \cdot x$ is bounded below.
    \end{enumerate}
\end{fact}
\begin{proof}
    That $(2) \Leftrightarrow (3)$ follows from the previous fact. We now prove $(1) \Rightarrow (2)$. The $(2) \Rightarrow (1)$ direction also has a short proof, based on a separating hyperplane for $\alpha$ and $\Newt(p)$ whenever $\alpha \not\in \Newt(p)$. The details can be found in Fact 2.18 of \cite{anari2017generalization}.
    
    Suppose that $\alpha \in \Newt(p)$. So, $\alpha = \sum_{\mu \in S} c_\mu \mu$, where $S \subset \supp(p)$, $c_\mu > 0$, and $\sum_{\mu \in S} c_\mu = 1$. Using the AM-GM inequality and the fact that the coefficients of $p$ are non-negative, we have the following for $x \in \R_+^n$:
    \[
        p(x) \geq \sum_{\mu \in S} p_\mu x^\mu = \sum_{\mu \in S} c_\mu \frac{p_\mu x^\mu}{c_\mu} \geq \prod_{\mu \in S} \left(\frac{p_\mu x^\mu}{c_\mu}\right)^{c_\mu} = x^\alpha \prod_{\mu \in S} \left(\frac{p_\mu}{c_\mu}\right)^{c_\mu}.
    \]
    This then implies:
    \[
        \cpc_\alpha(p) = \inf_{x > 0} \frac{p(x)}{x^\alpha} \geq \prod_{\mu \in S} \left(\frac{p_\mu}{c_\mu}\right)^{c_\mu} > 0.
    \]
    
\end{proof}

Due to the previous result, we will only ever consider values of $\alpha$ which are in the Newton polytope of the relevant polynomials. Other $\alpha$ can be considered but most results will then become trivial. That said, we will often make this assumption about $\alpha$ without explicitly stating it.

The next result emulates Proposition \ref{rsclosure} (the basic real stability preservers) by giving a collection of basic capacity preserving operators. Note that these results are either equalities, or give something of the form $\cpc(T(p)) \geq c_T \cdot \cpc(p)$ for various operators $T$.


\begin{proposition}[Basic capacity preservers] \label{capclosure}
    For $p,q \in \R^+_\lambda[x_1,...,x_n]$ and $\alpha,\beta \in \R_+^n$, we have:
    \begin{enumerate}
        \item Scaling: $\cpc_\alpha(bp) = b \cdot \cpc_\alpha(p)$ for $b \in \R_+$
        \item Product: $\cpc_{\alpha + \beta}(pq) \geq \cpc_\alpha(p) \cpc_\beta(q)$
        \item Disjoint product: $\cpc_{(\alpha,\beta)}(p(x)q(z)) = \cpc_\alpha(p) \cpc_\beta(q)$
        \item Evaluation: $\cpc_{(\alpha_1,...,\alpha_{n-1})}(p(x_1,...,x_{n-1}, y_n)) \geq y_n^{\alpha_n} \cpc_\alpha(p)$ for $y_n \in \R_+$
        \item External field: $\cpc_\alpha(p(cx)) = c^\alpha \cpc_\alpha(p)$ for $c \in \R_+^n$
        \item Inversion: $\cpc_{(\lambda-\alpha)}(x^\lambda p(x_1^{-1}, ..., x_n^{-1})) = \cpc_\alpha(p)$
        \item Concavity: $\cpc_\alpha(bp + cq) \geq b \cdot \cpc_\alpha(p) + c \cdot \cpc_\alpha(q)$ for $b,c \in \R_+$
        \item Diagonalization: $\cpc_{\sum \alpha_k}(p(x,...,x)) \geq \cpc_\alpha(p)$
        \item Symmetric diagonalization: $\cpc_{n \cdot \alpha_0}(p(x,...,x)) = \cpc_\alpha(p)$ if $\alpha = (\alpha_0,...,\alpha_0)$ and $p$ is symmetric
        \item Homogenization: $\cpc_{(\alpha, \lambda-\alpha)}(\Hmg_\lambda(p)) = \cpc_\alpha(p)$
    \end{enumerate}
\end{proposition}
\begin{proof}
    Symmetric diagonalization is the only nontrivial property, and it is a consequence of the AM-GM inequality. First of all, we automatically have (the diagonalization inequality):
    \[
        \cpc_{n \cdot \alpha_0}(p(x,...,x)) = \inf_{x > 0} \frac{p(x,...,x)}{x^{\alpha_0} \cdots x^{\alpha_0}} \geq \inf_{x > 0} \frac{p(x_1,...,x_n)}{x_1^{\alpha_0} \cdots x_n^{\alpha_0}} = \cpc_\alpha(p).
    \]
    For the other direction, fix $x \in \R_+^n$ and let $y := (x_1 \cdots x_n)^{1/n}$. Further, let $S(p)$ denote the symmetrization of $p$. For any $\mu \in \Z_+^n$, the AM-GM inequality gives:
    \[
    \begin{split}
        S(x^\mu) &= \frac{1}{n!} \sum_{\sigma \in S_n} x_{\sigma(1)}^{\mu_1} \cdots x_{\sigma(n)}^{\mu_n} \\
            &\geq \left(\prod_{\sigma \in S_n} x_{\sigma(1)}^{\mu_1} \cdots x_{\sigma(n)}^{\mu_n}\right)^{1/n!} \\
            &= \left(\prod_{j,k} x_j^{\mu_k}\right)^{1/n} = y^{\mu_1} \cdots y^{\mu_n}.
    \end{split}
    \]
    Additionally, $x^\alpha = x_1^{\alpha_0} \cdots x_n^{\alpha_0} = y^{n \cdot \alpha_0}$. Since $p$ is symmetric, we then have the following:
    \[
        \frac{p(x)}{x^\alpha} = \frac{S(p)(x)}{x^\alpha} = \sum_{\mu \in \supp(p)} p_\mu \frac{S(x^\mu)}{x^\alpha} \geq \sum_{\mu \in \supp(p)} p_\mu \frac{y^{\mu_1} \cdots y^{\mu_n}}{y^{n \cdot \alpha_0}} = \frac{p(y,...,y)}{y^{n \cdot \alpha_0}}.
    \]
    That is, for any $x \in \R_+^n$, there is a $y \in \R_+$ such that $\frac{p(x)}{x^\alpha} \geq \frac{p(y,...,y)}{y^{n \cdot \alpha_0}}$. Therefore:
    \[
        \cpc_\alpha(p) \geq \cpc_{n \cdot \alpha_0}(p(x,...,x)).
    \]
    This completes the proof.
\end{proof}

Many of these operations are similar to those that preserve real stability. This is to be expected, as we hope to combine the two theories. In this vein, we now discuss the capacity preservation properties of the polarization operator. As it does for real stability preservers, polarization will play a crucial role in working out the theory of capacity preservers. To state this result, we define the \emph{polarization of the vector} $\alpha$ as follows, where each value $\frac{\alpha_k}{\lambda_k}$ shows up $\lambda_k$ times:
\[
    \Pol^\lambda(\alpha) := \left(\frac{\alpha_1}{\lambda_1}, ..., \frac{\alpha_1}{\lambda_1}, \frac{\alpha_2}{\lambda_2}, ..., \frac{\alpha_2}{\lambda_2}, ..., \frac{\alpha_n}{\lambda_n}, ..., \frac{\alpha_n}{\lambda_n}\right).
\]

\begin{proposition} \label{cappol}
    Given $p \in \R_+^\lambda[x_1,...,x_n]$ and $\alpha \in \R_+^n$, we have that $\cpc_{\Pol^\lambda(\alpha)}(\Pol^\lambda(p)) = \cpc_\alpha(p)$.
\end{proposition}
\begin{proof}
    We essentially apply the diagonalization property to each variable in succession. Specifically, we have:
    \[
        \begin{split}
            \cpc_\alpha(p) &= \inf_{y_1,...,y_{n-1} > 0} \frac{1}{y_1^{\alpha_1} \cdots y_{n-1}^{\alpha_{n-1}}} \inf_{x_n > 0} \frac{p(y_1,...,y_{n-1}, x_n)}{x_n^{\alpha_n}} \\
                &= \inf_{y_1,...,y_{n-1} > 0} \frac{1}{y_1^{\alpha_1} \cdots y_{n-1}^{\alpha_{n-1}}} \cpc_{\alpha_n}(p(y_1,...,y_{n-1}, x_n)) \\
                &= \inf_{y_1,...,y_{n-1} > 0} \frac{1}{y_1^{\alpha_1} \cdots y_{n-1}^{\alpha_{n-1}}} \cpc_{\Pol^{\lambda_n}(\alpha_n)}(\Pol^{\lambda_n}(p(y_1,...,y_{n-1}, \cdot))).
        \end{split}
    \]
    By now rearranging the $\inf$'s in the last expression above, we can let $\inf_{y_{n-1} > 0}$ be the inner-most $\inf$. We can then apply the above argument again, and this will work for every $y_k$ in succession. At the end of this process, we obtain:
    \[
    \begin{split}
        \cpc_\alpha(p) &= \cpc_{(\Pol^{\lambda_1}(\alpha_1), ..., \Pol^{\lambda_n}(\alpha_n))}(\Pol^{\lambda_1} \circ \cdots \circ \Pol^{\lambda_n}(p)) \\
            &= \cpc_{\Pol^\lambda(\alpha)}(\Pol^\lambda(p)).
    \end{split}
    \]
\end{proof}

Note that the two main results on polarization---capacity preservation and real stability preservation---imply that we only really need to prove our results in the multiaffine case (i.e., where polynomials are of degree at most 1 in each variable). We will make use of this reduction when we prove our technical results in \S\ref{mainineq}.

Finally before moving on, we give one basic capacity calculation which will prove extremely useful to us almost every time we want to compute capacity.

\begin{lemma} \label{capcompute}
    For $c,\alpha \in \R_+^n$ and $m := \sum_k \alpha_k$, we have the following:
    \[
        \cpc_\alpha((c \cdot x)^m) \equiv \cpc_\alpha\bigg(\big(\sum_k c_k x_k\big)^m\bigg) = \left(\frac{mc}{\alpha}\right)^{\alpha}.
    \]
\end{lemma}
\begin{proof}
    Note first that:
    \[
        \cpc_\alpha((c \cdot x)^m) = \left(\cpc_{\frac{\alpha}{m}}(c \cdot x)\right)^m.
    \]
    To compute $\cpc_{\frac{\alpha}{m}}(c \cdot x)$, we use calculus. Let $\beta := \frac{\alpha}{m}$, and for now we assume that $\beta > 0$ and $c > 0$ strictly. We have:
    \[
        \partial_{x_k} \left(\frac{c \cdot x}{x^\beta}\right) = \frac{x^\beta c_k - \beta_k x^{\beta - \delta_k} (c \cdot x)}{x^{2\beta}} = \frac{x_k c_k - \beta_k (c \cdot x)}{x^{\beta+\delta_k}}.
    \]
    That is, the gradient of $\frac{c \cdot x}{x^\beta}$ is the 0 vector whenever $\frac{c_k}{\beta_k} x_k = c \cdot x$ for all $k$. And in fact, any vector satisfying those conditions should minimize $\frac{c \cdot x}{x^\beta}$, by homogeneity. Since $\sum_k \beta_k = 1$, the vector $x_k := \frac{\beta_k}{c_k}$ satisfies the conditions. This implies:
    \[
        \cpc_\beta(c \cdot x) = \frac{c \cdot (\beta/c)}{(\beta/c)^\beta} = \left(\frac{c}{\beta}\right)^\beta.
    \]
    Therefore:
    \[
        \cpc_\alpha((c \cdot x)^m) = \left(\frac{c}{\beta}\right)^{m\beta} = \left(\frac{mc}{\alpha}\right)^{\alpha}.
    \]
\end{proof}

\section{Applications of Capacity Preservers} \label{applications}

We now formally state and discuss our main results and their applications. As mentioned above, we will emulate the Borcea-Br\"and\'en characterization for capacity preservers. Further, we will also demonstrate how our results encapsulate many of the previous results regarding capacity. With this in mind, we first give our main capacity preservation results: one for bounded degree operators a la Theorem \ref{bbbounded}, and one for unbounded degree operators a la Theorem \ref{bbtrans}. Notice that the unbounded degree case is something like a limit of the bounded degree case: the scalar $\frac{\alpha^\alpha (\lambda-\alpha)^{\lambda-\alpha}}{\lambda^\lambda}$ is approximately $\left(\frac{\alpha}{\lambda}\right)^\alpha e^{-\alpha}$ as $\lambda \to \infty$. (The proof of Theorem \ref{boundtrans} shows why the extra $\lambda^{-\alpha}$ factor disappears.)

\begin{theorem}[= Theorem \ref{mainthmbounded}; Bounded degree case]
    Fix a linear operator $T: \R_+^\lambda[x_1,...,x_n] \to \R_+^\gamma[x_1,...,x_m]$ with real stable symbol. For any $\alpha \in \R_+^n$, any $\beta \in \R_+^m$, and any real stable $p \in \R_+^\lambda[x_1,...,x_n]$ we have:
    \[
        \frac{\cpc_\beta(T(p))}{\cpc_\alpha(p)} \geq \frac{\alpha^\alpha (\lambda-\alpha)^{\lambda-\alpha}}{\lambda^\lambda} \cpc_{(\alpha,\beta)}(\Symb^\lambda(T)).
    \]
    Further, this bound is tight for fixed $T$, $\alpha$, and $\beta$.
\end{theorem}

\begin{theorem}[= Theorem \ref{mainthmtrans}; Unbounded degree case]
    Fix a linear operator $T: \R_+[x_1,...,x_n] \to \R_+[x_1,...,x_m]$ with real stable symbol. Then for any $\alpha \in \R_+^n$, any $\beta \in \R_+^m$, and any real stable $p \in \R_+[x_1,...,x_n]$ we have:
    \[
        \frac{\cpc_\beta(T(p))}{\cpc_\alpha(p)} \geq \alpha^\alpha e^{-\alpha} \cpc_{(\alpha,\beta)}(\Symb^\infty(T)).
    \]
    Further, this bound is tight for fixed $T$, $\alpha$, and $\beta$.
\end{theorem}

Note that by Theorems \ref{bbbounded} and \ref{bbtrans}, the above theorems apply to real stability preservers of rank greater than 2 (see Corollaries \ref{maincorbounded} and \ref{maincortrans}).



\subsection{Gurvits' Theorem}

With these results in hand, we now reprove Gurvits' theorem and discuss its importance. Gurvits' original proof of this fact was not very complicated, and our proof will be similar in this regard. This is of course what makes capacity and real stability more generally so intriguing: answers to seemingly hard questions follow from a few basic computations on polynomials.

\begin{theorem}[Gurvits] \label{gurvits}
    For real stable $p \in \R_+^\lambda[x_1,...,x_n]$ we have:
    \[
        \frac{\cpc_{(1^{n-1})}\left(\left.\partial_{x_k}p\right|_{x_k=0}\right)}{\cpc_{(1^n)}(p)} \geq \left(\frac{\lambda_k - 1}{\lambda_k}\right)^{\lambda_k - 1}.
    \]
\end{theorem}
\begin{proof}
    We apply Theorem \ref{mainthmbounded} above for $T := \left.\partial_{x_k}\right|_{x_k=0}$, $\alpha := (1^n)$, and $\beta := (1^{n-1})$. To do this we need to compute the right-hand side of the expression in Theorem \ref{mainthmbounded}, making use of properties from Proposition \ref{capclosure}. We have:
    \[
        \begin{split}
            \frac{\alpha^\alpha (\lambda-\alpha)^{\lambda-\alpha}}{\lambda^\lambda} &\cpc_{(\alpha,\beta)}(\Symb^\lambda(T)) \\
                &= \frac{(\lambda-1)^{\lambda-1}}{\lambda^\lambda} \cpc_{(1^n,1^{n-1})}\left(\left.\partial_{x_k}(1+xz)^\lambda\right|_{x_k=0}\right) \\
                &= \frac{(\lambda-1)^{\lambda-1}}{\lambda^\lambda} \cpc_{(1^n,1^{n-1})}\left(\lambda_k z_k \prod_{j \neq k} (1+x_jz_j)^{\lambda_j}\right) \\
                &= \frac{(\lambda-1)^{\lambda-1}}{\lambda^\lambda} \lambda_k \prod_{j \neq k} \cpc_{(1,1)}\left((1+x_jz_j)^{\lambda_j}\right).
        \end{split}
    \]
    Note that $\cpc_{(1,1)}((1+x_jz_j)^{\lambda_j}) = \cpc_1((1+x_j)^{\lambda_j})$. Using the homogenization property and Lemma \ref{capcompute}, we then have:
    \[
        \cpc_1((1+x_j)^{\lambda_j}) = \cpc_{(1,\lambda_j-1)}((x_j + y_j)^{\lambda_j}) = \lambda_j \left(\frac{\lambda_j}{\lambda_j - 1}\right)^{\lambda_j - 1}.
    \]
    Therefore:
    \[
        \begin{split}
            \frac{\alpha^\alpha (\lambda-\alpha)^{\lambda-\alpha}}{\lambda^\lambda} &\cpc_{(\alpha,\beta)}(\Symb^\lambda(T)) \\
                &= \frac{(\lambda-1)^{\lambda-1}}{\lambda^\lambda} \lambda_k \prod_{j \neq k} \cpc_{(1,1)}\left((1+x_jz_j)^{\lambda_j}\right) \\
                &= \frac{(\lambda-1)^{\lambda-1}}{\lambda^\lambda} \lambda_k \prod_{j \neq k} \lambda_j \left(\frac{\lambda_j}{\lambda_j - 1}\right)^{\lambda_j - 1} \\
                &= \left(\frac{\lambda_k - 1}{\lambda_k}\right)^{\lambda_k - 1}.
        \end{split}
    \]
\end{proof}

This proof will serve as a good baseline for other applications of our main theorems. Roughly speaking, most applications will make use of Lemma \ref{capcompute} and the properties of Proposition \ref{capclosure} in interesting ways. And often, the inequalities obtained will directly translate to various combinatorial statements.

Specifically, what sorts of combinatorial statements can be derived from Gurvits' theorem? The most well known are perhaps Schrijver's theorem and the Van der Waerden bound on the permanent (see \cite{gurvits2008van}). What forms the link between capacity and combinatorial objects like doubly stochastic matrices and perfect matchings is the following polynomial defined for a given matrix $M$:
\[
    p_M(x) := \prod_i \sum_j m_{ij} x_j.
\]
Note that this polynomial is real stable whenever the entries of $M$ are nonnegative. The following is then quite suggestive.

\begin{lemma}[Gurvits]
    If $M$ is a doubly stochastic matrix, then $\cpc_{(1^n)}(p_M) = 1$.
\end{lemma}
\begin{proof}
    Follows from Fact \ref{fact:marginals}.
\end{proof}

For most of the arguments, one considers $p = p_M$ and $T$ such that $T(p)$ computes the desired quantity related to $M$. The above theorems then give something like:
\[
    \text{desired quantity} = \frac{\cpc(T(p_M))}{\cpc(p_M)} \geq \text{constant depending on $T$ but not on $M$}.
\]
This gives us a bound on the desired quantity (e.g. perfect matchings or the permanent) for any $M$, so long as we can compute the capacity of $\Symb(T)$.

In addition to these types of inequalities, Gurvits also demonstrates how his theorem implies similar results for ``doubly stochastic'' $n$-tuples of matrices (a conjecture due to Bapat \cite{bapat1989mixed}). In fact, this notion of doubly stochastic aligns with a generalized notion used recently in \cite{garg2016deterministic},\cite{burgisser2017alternating}. In those papers, doubly stochastic matrices and other similar objects play a crucial role in defining certain important orbits of actions on tuples of matrices. Specifically in \cite{garg2016deterministic} (or \cite{garg2015operator}), a version of this idea was used to produce a polynomial time algorithm for the noncommutative polynomial identity testing problem. A certain notion of capacity for matrices was quite important in the analysis of their algorithms.

\subsection{Imperfect Matchings and Biregular Graphs} \label{sec:biregular_graphs}

The most important application of our results is a new proof of a bound on size-$k$ matchings of a biregular bipartite graph, due to Csikv\'ari \cite{csikvari2014lower}. This result is a generalization of Schrijver's bound, 
and it also settled and strengthened the Friedland macthing conjecture \cite{friedland2008number}. We first state Csikv\'ari's results, in a form more amenable to the notation of this paper.

\begin{theorem}[Csikv\'ari] \label{csikvari}
    Let $G$ be an $(a,b)$-biregular bipartite graph with $(m,n)$-bipartitioned vertices (so that $am = bn$ is the number of edges of $G$). Then the number of size-$k$ matchings of $G$ is bounded as follows:
    \[
        \mu_k(G) \geq \binom{n}{k}(ab)^k \frac{m^m(ma-k)^{ma-k}}{(ma)^{ma}(m-k)^{m-k}}.
    \]
\end{theorem}

Notice that this immediately implies the following bound for regular bipartite graphs.

\begin{corollary}[Csikv\'ari]
    Let $G$ be a $d$-regular bipartite graph with $2n$ vertices. Then:
    \[
        \mu_k(G) \geq \binom{n}{k} d^k \left(\frac{nd-k}{nd}\right)^{nd-k} \left(\frac{n}{n-k}\right)^{n-k}.
    \]
\end{corollary}

To prove these results, we first need to generalize Gurvits' capacity lemma for doubly stochastic matrices. Specifically we want to be able to handle \emph{$(a,b)$-stochastic matrices}, which are matrices with row sums equal to $a$ and columns sums equal to $b$. We care about such matrices, because the bipartite adjacency matrix of a $(a,b)$-biregular graph is $(a,b)$-stochastic. Note that if $M$ is an $(a,b)$-stochastic matrix which is of size $m \times n$, then $am = bn$.

\begin{lemma}
    If $M$ is an $(a,b)$-stochastic matrix, then $\cpc_{(\frac{m}{n},...,\frac{m}{n})}(p_M) = a^m$.
\end{lemma}
\begin{proof}
    Follows from Fact \ref{fact:marginals}.
\end{proof}

We also need a linear operator which computes the number of size-$k$ matchings of an $(a,b)$-biregular bipartite graph. In fact when $M$ is the bipartite adjacency matrix of $G$, we have the following:
\[
    a^{m-k} \mu_k(G) = \sum_{S \in \binom{[n]}{k}} \partial_x^S p_M(1) = \cpc_\varnothing \Bigg(\sum_{S \in \binom{[n]}{k}} \partial_x^S p_M(1)\Bigg).
\]
Note that each differential operator in the sum picks out a disjoint collection of $k \times k$ subpermutations of the matrix $M$. After applying each differential operator, we are left with terms which are products of $m-k$ remaining linear forms from $p_M$. Plugging in 1 then gives $a^{m-k}$ (since row sums are $a$), and this is why that factor appears above.

Next, we need to prove that we can apply Theorem \ref{mainthmbounded} to the operator $T := \sum_{S \in \binom{[n]}{k}} \left.\partial_x^S\right|_{x=1}$. We choose $\lambda = (b,...,b)$ here because $(a,b)$-regularity of $G$ implies every variable will be of degree $b$ in the polynomial $p_M$ (where $M$ is the bipartite adjacency matrix of $G$). That is, some of the degree information of $G$ is encoded as the degree of the associated polynomial $p_M$. This same thing was done in \cite{gurvits2006hyperbolic} and \cite{friedland2008lower}, and this further demonstrates how capacity bounds can combine an interesting mix of analytic and combinatorial information.

\begin{lemma}
    The operator $T := \sum_{S \in \binom{[n]}{k}} \left.\partial_x^S\right|_{x=1}$ has real stable symbol.
\end{lemma}
\begin{proof}
    Here the input polynomial space is $R_+^{(b,...,b)}[x_1,...,x_n]$, since degree is determined by the column sums. Denoting $\lambda := (b,...,b)$, we compute $\Symb^\lambda(T)$:
    \[
    \begin{split}
        T[(1+xz)^\lambda] &= \sum_{S \in \binom{[n]}{k}} \left.\partial_x^S\right|_{x=1} (1+xz)^\lambda \\
            &= \sum_{S \in \binom{[n]}{k}} b^k z^S (1+z)^{\lambda - S} \\
            &= b^k (1+z)^{\lambda-1} \sum_{S \in \binom{[n]}{k}} z^S (1+z)^{1-S}.
    \end{split}
    \]
    Notice that $\sum_{S \in \binom{[n]}{k}} z^S (1+z)^{1-S} = \binom{n}{k}\Pol^n(x^k(1+x)^{n-k})$, which is real stable by Proposition \ref{rspol}.
\end{proof}

Applying Theorem \ref{mainthmbounded} now shows us the way toward the rest of the proof. Denoting $\lambda := (b,...,b)$ and $\alpha := (\frac{m}{n},...,\frac{m}{n})$, we now have:
\[
    \begin{split}
        a^{m-k} \mu_k(G) &= \sum_{S \in \binom{[n]}{k}} \partial_x^S p_M(1) \\
            &\geq \frac{\alpha^\alpha (\lambda-\alpha)^{\lambda-\alpha}}{\lambda^\lambda} \cpc_\alpha(p_M) \cpc_{(\alpha,\varnothing)}(\Symb^\lambda(T)) \\
            &= \left(\frac{(\frac{m}{n})^{\frac{m}{n}} (b - \frac{m}{n})^{b - \frac{m}{n}}}{b^b}\right)^n a^m \cpc_\alpha(\Symb^\lambda(T)) \\
            &= \frac{(ma)^m (nb-m)^{nb-m}}{(nb)^{nb}} \cpc_\alpha(\Symb^\lambda(T)).
    \end{split}
\]

So the last computation we need to make is that of $\cpc_\alpha(\Symb^\lambda(T))$. Fortunately since $\Symb^\lambda(T)$ is symmetric and $\alpha = (\frac{m}{n},...,\frac{m}{n})$, we can apply the symmetric diagonalization property to simplify this computation. Using our previous computation of $\Symb^\lambda(T)$, this gives:
\[
\begin{split}
    \cpc_{(\frac{m}{n},...,\frac{m}{n})}(\Symb^\lambda(T)) &= \cpc_m \left(b^k \binom{n}{k} z^k (1+z)^{nb-k}\right) \\
        &= b^k \binom{n}{k} \cpc_m(z^k (1+z)^{nb-k}).
\end{split}
\]
The remaining capacity computation then follows from homogenization and Lemma \ref{capcompute}:
\[
\begin{split}
    \cpc_m(z^k (1+z)^{nb-k}) &= \inf_{z > 0} \frac{z^k (1+z)^{nb-k}}{z^m} \\
        &= \cpc_{m-k}((1+z)^{nb-k}) \\
        &= \left(\frac{nb-k}{m-k}\right)^{m-k} \left(\frac{nb-k}{nb-m}\right)^{nb-m}.
\end{split}
\]
Putting all of these computations together and recalling $ma = nb$ gives:
\[
    \begin{split}
        \mu_k(G) &\geq a^{k-m} \frac{(ma)^m (nb-m)^{nb-m}}{(nb)^{nb}} b^k \binom{n}{k} \left(\frac{nb-k}{m-k}\right)^{m-k} \left(\frac{nb-k}{nb-m}\right)^{nb-m} \\
            &= \binom{n}{k} a^{k-m} b^k \frac{(ma)^m(nb-k)^{nb-k}}{(nb)^{nb}(m-k)^{m-k}} \\
            &= \binom{n}{k} (ab)^k \frac{m^m(ma-k)^{ma-k}}{(ma)^{ma}(m-k)^{m-k}}.
    \end{split}
\]
This is precisely the desired inequality.

\subsection{Differential Operators in General}

We now give general capacity preservation bounds for stability preservers which are differential operators. This was done in \cite{anari2017generalization} for differential operators which preserve real stability on input polynomials of all degrees. Here, we restrict to those operators which only preserve real stability for polynomials of some fixed bounded degree. That said, consider the following bilinear operator:
\[
    (p \boxplus^\lambda q)(x) := \sum_{0 \leq \mu \leq \lambda} (\partial_x^\mu p)(x) (\partial_x^{\lambda-\mu} q)(0).
\]
It is straightforward to see that by fixing $q$, one can construct any constant coefficient differential operator on $\R^\lambda[x_1,...,x_n]$. And it turns out that if $q$ is real stable, then $(\cdot \boxplus^\lambda q)(x)$ is a real stability preserver.

It turns out that more is true, however. The operator $\boxplus^\lambda$ can actually be applied to polynomials in $\R^{(\lambda,\lambda)}[x_1,...,x_n,y_1,...,y_n]$ by considering this polynomial space as a tensor product of polynomial spaces. More concretely, we specify how this operator acts on the monomial basis:
\[
    \boxplus^\lambda: x^\mu y^\nu \mapsto x^\mu \boxplus^\lambda x^\nu.
\]
We can then compute the symbol of this operator:
\[
    \Symb[\boxplus^\lambda] = (1+xz)^\lambda \boxplus^\lambda (1+xw)^\lambda = (z + w + zwx)^\lambda = (zw)^\lambda (x + z^{-1} + w^{-1})^\lambda.
\]
Note that $\Symb[\boxplus^\lambda](z,w,-x)$ is real stable, and so $\boxplus^\lambda$ preserves real stability by Theorem \ref{bbbounded}.

With this, we compute the capacity for $\lambda = \delta_1 = (1,0,...,0)$:
\[
\begin{split}
    \cpc_{(\alpha,\beta,\gamma)} (z + t + ztx) &= \inf_{x,z,t > 0} \frac{z+t+ztx}{x^\alpha z^\beta t^\gamma} \\
        &= \inf_{x,z,t > 0} \frac{t^{-1}+z^{-1}+x}{x^\alpha z^{\beta-1} t^{\gamma-1}} \\
        &= \inf_{x,z,t > 0} \frac{t+z+x}{x^\alpha z^{1-\beta} t^{1-\gamma}}.
\end{split}
\]
Note that $(\alpha,\beta,\gamma)$ is in the Newton polytope of $(z+t+ztx)$ iff $\alpha = \beta+\gamma-1$. By Lemma \ref{capcompute}, we have:
\[
    \cpc_{(\alpha,\beta,\gamma)} (z + t + ztx) = \frac{1}{\alpha^\alpha (1-\beta)^{1-\beta} (1-\gamma)^{1-\gamma}}.
\]
We now generalize this to general $\lambda$, supposing $\alpha = \beta + \gamma - \lambda$:
\[
    \begin{split}
        \cpc_{(\alpha,\beta,\gamma)}&((z+t+ztx)^\lambda) \\
            &= \prod_{j=1}^n \left((\alpha_j/\lambda_j)^{\alpha_j/\lambda_j} (1-\beta_j/\lambda_j)^{1-\beta_j/\lambda_j} (1-\gamma_j/\lambda_j)^{1-\gamma_j/\lambda_j}\right)^{-\lambda_j} \\
            &= \prod_{j=1}^n (\alpha_j/\lambda_j)^{-\alpha_j} (1-\beta_j/\lambda_j)^{\beta_j-\lambda_j} (1-\gamma_j/\lambda_j)^{\gamma_j-\lambda_j} \\
            &= \alpha^{-\alpha} (\lambda-\beta)^{\beta-\lambda} (\lambda-\gamma)^{\gamma-\lambda} \lambda^{\alpha-\beta-\gamma+2\lambda} \\
            &= \frac{\lambda^\lambda}{\alpha^\alpha (\lambda-\beta)^{\lambda-\beta} (\lambda-\gamma)^{\lambda-\gamma}}.
    \end{split}
\]
Applying Theorem \ref{mainthmbounded}, we get:
\[
    \begin{split}
        \cpc_\alpha(p \boxplus^\lambda q) &\geq \frac{\beta^\beta \gamma^\gamma (\lambda-\beta)^{\lambda-\beta} (\lambda-\gamma)^{\lambda-\gamma} \cdot \lambda^\lambda}{\lambda^\lambda \lambda^\lambda \cdot \alpha^\alpha (\lambda-\beta)^{\lambda-\beta} (\lambda-\gamma)^{\lambda-\gamma}} \cdot \cpc_\beta(p) \cpc_\gamma(q) \\
            &= \frac{\beta^\beta \gamma^\gamma}{\alpha^\alpha \lambda^\lambda} \cpc_\beta(p) \cpc_\gamma(q).
    \end{split}
\]
Again, this is all under the assumption that $\alpha = \beta + \gamma - \lambda$. (We will be outside the Newton polytope otherwise, and so the result in that case will be trivial.) We state the result of this discussion as follows. Note that is can be seen as a sort of multiplicative reverse triangle inequality for capacity of differential operators.

\begin{corollary}
    Let $p$ and $q$ be two real stable polynomials of degree $\lambda$ with positive coefficients. We have:
    \[
        (\alpha^\alpha \cpc_\alpha(p \boxplus^\lambda q)) \geq \frac{1}{\lambda^\lambda} \cdot (\beta^\beta \cpc_\beta(p)) \cdot (\gamma^\gamma \cpc_\gamma(q)).
    \]
\end{corollary}

With this, we have given tight capacity bounds for all differential operators on polynomials of at most some fixed bounded degree. Note that root bounds of this form are given in \cite{finiteconvolutions} by Marcus, Spielman, and Srivastava, and these bounds are related to those obtained in their proof of the Kadison-Singer conjecture in \cite{marcus2015interlacing}. It is an open and interesting question whether or not capacity can be utilized to bound the roots of polynomials.


\section{The Main Inequalities} \label{mainineq}

We now discuss our main results and the inequalities we use to obtain them. These inequalities are bounds on certain inner products applied to polynomials. The most basic of these is the main result from \cite{anari2017generalization}, which applies to multiaffine polynomials. We extend their methods to obtain bounds on polynomials of all degrees. Finally, a limiting argument implies bounds for the $\mathcal{LP}_+$ class. This last bound can also be found in \cite{anari2017generalization}, but the proof we give here is simpler and makes clearer the connection between these inequalities and the Borcea-Br\"and\'en characterization.

\subsection{Inner Product Bounds, Bounded Degree}

For polynomials of some fixed bounded degree, we consider the following inner product.

\begin{definition}
    For fixed $\lambda \in \Z_+^n$ and $p,q \in \R^\lambda[x_1,...,x_n]$, define:
    \[
        \langle p, q \rangle^\lambda := \sum_{0 \leq \mu \leq \lambda} \binom{\lambda}{\mu}^{-1} p_\mu q_\mu.
    \]
\end{definition}

As mentioned above, Anari and Oveis~Gharan prove a bound on the above inner product for multiaffine polynomials in \cite{anari2017generalization}, and we state their result here without proof. We note though that the proof is essentially a consequence of the strong Rayleigh inequalities for real stable polynomials, which we now state. These fundamental inequalities (due to Br\"and\'en) should be seen as log-concavity conditions, and this intuition extends to all the inner product bounds we state here. And this intuition is not without evidence: the connection of capacity to the Alexandrov-Fenchel inequalities (see \cite{gurvits2006hyperbolic}), as well as to matroids and log-concave polynomials (see \cite{gurvits2009multivariate} and more recently \cite{anari2018log}), has been previously noted.

\begin{proposition}[Strong Rayleigh inequalities \cite{strongrayleigh}]
    For any real stable $p \in \R^{(1^n)}$ and any $i,j \in [n]$, we have the following inequality pointwise on all of $\R^n$:
    \[
        (\partial_{x_i} p) \cdot (\partial_{x_j} p) \geq p \cdot (\partial_{x_i}\partial_{x_j} p).
    \]
\end{proposition}

We now state the Anari-Oveis~Gharan bound for multiaffine polynomials. They also prove a weaker bound on polynomials of any degree, but we will discuss this later.

\begin{theorem}[Anari-Oveis~Gharan]
    Let $p,q \in \R_+^{(1^n)}[x_1,...,x_n]$ be real stable. Then for any $\alpha \in \R_+^n$ we have:
    \[
        \langle p, q \rangle^{(1^n)} \geq \alpha^\alpha (1-\alpha)^{1-\alpha} \cpc_\alpha(p) \cpc_\alpha(q).
    \]
\end{theorem}

In this paper, we generalize this to polynomials of degree $\lambda$ as follows. Note that this result is strictly stronger than the bound obtained in \cite{anari2017generalization} for the non-multiaffine case.

\begin{theorem} \label{boundbounded}
    Let $p,q \in \R_+^\lambda[x_1,...,x_n]$ be real stable. Then for any $\alpha \in \R_+^n$ we have:
    \[
        \langle p, q \rangle^{\lambda} \geq \frac{\alpha^\alpha (\lambda-\alpha)^{\lambda-\alpha}}{\lambda^\lambda} \cpc_\alpha(p) \cpc_\alpha(q).
    \]
\end{theorem}

The proof of this is essentially due to the fact that both $\langle \cdot, \cdot \rangle^\lambda$ and capacity interact nicely with polarization. We have already explicated the connection between capacity and polarization (see Proposition \ref{cappol}), and we now demonstrate how these inner products fit in.


\begin{lemma}
    Given $p,q \in \R^\lambda[x_1,...,x_n]$, we have:
    \[
        \langle p, q \rangle^\lambda = \left\langle \Pol^\lambda(p), \Pol^\lambda(q) \right\rangle^{(1^\lambda)}.
    \]
\end{lemma}
\begin{proof}
    We compute this equality on a basis in the univariate case. The result then follows since $\Pol^\lambda$ is a composition of polarizations on each variable of $p$. For $0 \leq k \leq m$ we have:
    \[
    \begin{split}
        \big\langle\Pol^m(x^k), \Pol^m(x^k)\big\rangle^{(1^m)} &= \binom{m}{k}^{-2} \sum_{S \in \binom{[m]}{k}} \langle x^S, x^S \rangle^{(1^m)} \\
            &= \binom{m}{k}^{-1} = \langle x^k, x^k \rangle^m.
    \end{split}
    \]
\end{proof}

The proof of Theorem \ref{boundbounded} then essentially follows from this algebraic identity.

\begin{proof}[Proof of Theorem \ref{boundbounded}.]
    Suppose that $p,q \in \R_+^\lambda[x_1,...,x_n]$ are real stable polynomials. Then $\Pol^\lambda(p)$ and $\Pol^\lambda(q)$ are real stable multiaffine polynomials by Proposition \ref{rspol}. We now use the multiaffine bound to prove the result for any $\alpha \in \R_+^n$. For simplicity, let $\beta := \Pol^\lambda(\alpha)$, where $\Pol^\lambda(\alpha)$ is originally defined in \S\ref{capss}. We have:
    \[
        \langle p, q \rangle^\lambda = \left\langle \Pol^\lambda(p), \Pol^\lambda(q) \right\rangle^{(1^\lambda)} \geq \beta^\beta (1-\beta)^{1-\beta} \cpc_\beta(\Pol^\lambda(p)) \cpc_\beta(\Pol^\lambda(q)).
    \]
    By Proposition \ref{cappol}, we have that $\cpc_\beta(\Pol^\lambda(p)) = \cpc_\alpha(p)$. So to complete the proof, we compute:
    \[
    \begin{split}
        \beta^\beta (1-\beta)^{1-\beta} &= \prod_{k=1}^n \prod_{j=1}^{\lambda_k} \left(\frac{\alpha_k}{\lambda_k}\right)^{\alpha_k/\lambda_k} \left(1 - \frac{\alpha_k}{\lambda_k}\right)^{1 - \alpha_k/\lambda_k} \\
            &= \prod_{k=1}^n \left(\frac{\alpha_k}{\lambda_k}\right)^{\alpha_k} \left(\frac{\lambda_k - \alpha_k}{\lambda_k}\right)^{\lambda_k - \alpha_k}.
    \end{split}
    \]
    This is precisely $\frac{\alpha^\alpha(\lambda-\alpha)^{\lambda-\alpha}}{\lambda^\lambda}$, which is what was claimed.
\end{proof}

\subsection{Inner Product Bounds, Unbounded Degree}

For general polynomials and power series in the $\mathcal{LP}_+$ class, we consider the following inner product.

\begin{definition}
    For $p,q \in \R[x_1,...,x_n]$ or power series in $x_1,...,x_n$, define:
    \[
        \langle p, q \rangle^\infty := \sum_{0 \leq \mu} \mu! p_\mu q_\mu.
    \]
    Note that this may not be well-defined for some power series.
\end{definition}

Consider the following power series in $x_1,...,x_n$, where $c_\mu \geq 0$:
\[
    f(x_1,...,x_n) = \sum_{0 \leq \mu} \frac{1}{\mu!} c_\mu x^\mu.
\]
Next consider the following weighted truncations of $f$:
\[
    f_\lambda(x) := \sum_{0 \leq \mu \leq \lambda} \binom{\lambda}{\mu} c_\mu x^\mu.
\]
If $f \in \mathcal{LP}_+[x_1,...,x_n]$, then $f_\lambda$ is real stable for all $\lambda$ and $f_\lambda(x/\lambda) \to f(x)$ uniformly on compact sets in $\C^n$ (see Theorem 5.1 in \cite{bb1}). The idea then is to limit capacity bounds for polynomials of some bounded degree to capacity bounds for general polynomials and functions in the $\mathcal{LP}_+$ class.

To do this, we need some kind of continuity result for capacity. Note that Fact \ref{newtoncap} implies $\cpc_\alpha(p)$ is \emph{not} continuous in $\alpha$ at the boundary of the Newton polytope of $p$. However, it turns out $\cpc_\alpha(p)$ is continuous in $p$, for the topology of uniform convergence on compact sets. This is discussed in \S\ref{continuity} more thoroughly, and we now state the main result from that section.

\newtheorem*{capcont}{Corollary \ref{capcont}}
\begin{capcont}
    Let $p_n$ be polynomials with nonnegative coefficients and $p$ analytic such that $p_n \to p$ uniformly on compact sets. For $\alpha \in \Newt(p)$, we have:
    \[
        \lim_{n \to \infty} \cpc_\alpha p_n = \cpc_\alpha p.
    \]
\end{capcont}

We now demonstrate the link between the bounded and unbounded degree inner products, and we will use this to obtain bounds on the latter via limiting.

\begin{lemma}
    Let $f$ and $f_\lambda$ be defined as above. For any $p \in \R_+[x_1,...,x_n]$ we have:
    \[
        \lim_{\lambda \to \infty} \langle f_\lambda, p \rangle^\lambda = \langle f, p \rangle^\infty.
    \]
\end{lemma}
\begin{proof}
    Letting $c_\mu$ denote the weighted coefficients of $f$ and $f_\lambda$ as above, we compute:
    \[
        \lim_{\lambda \to \infty} \langle f_\lambda, p \rangle^\lambda = \lim_{\lambda \to \infty} \sum_{0 \leq \mu \leq \lambda} c_\mu p_\mu = \sum_{0 \leq \mu} c_\mu p_\mu = \langle f, p \rangle^\infty.
    \]
    Notice that the limit here is guaranteed to exist, since $p$ has finite support.
\end{proof}


With this, we can bootstrap our capacity bound on $\langle \cdot, \cdot \rangle^\lambda$ to get a bound on $\langle \cdot, \cdot \rangle^\infty$. Notice here that we achieve the same bound as Anari and Oveis~Gharan in \cite{anari2017generalization}, albeit with a simpler proof.

\begin{theorem}[Anari-Oveis~Gharan] \label{boundtrans}
    Fix $f \in \mathcal{LP}_+[x_1,...,x_n]$ and any real stable $p \in \R_+[x_1,...,x_n]$. Then for any $\alpha \in \R_+^n$ we have:
    \[
        \langle f, p \rangle^\infty \geq \alpha^\alpha e^{-\alpha} \cpc_\alpha(f) \cpc_\alpha(p).
    \]
\end{theorem}
\begin{proof}
    As above, we write:
    \[
        f(x) = \sum_{0 \leq \mu} \frac{1}{\mu!} c_\mu x^\mu,
        ~~~~~~~~~~
        f_\lambda(x) = \sum_{0 \leq \mu \leq \lambda} \binom{\lambda}{\mu} c_\mu x^\mu.
    \]
    By the previous lemma, we have:
    \[
    \begin{split}
        \langle f, p \rangle^\infty &= \lim_{\lambda \to \infty} \langle f_\lambda, p \rangle^\lambda \\
            &\geq \lim_{\lambda \to \infty} \left[\frac{\alpha^\alpha(\lambda-\alpha)^{\lambda-\alpha}}{\lambda^\lambda} \cpc_\alpha(f_\lambda) \cpc_\alpha(p)\right] \\
            &= \alpha^\alpha \cpc_\alpha(p) \cdot \lim_{\lambda \to \infty} \left[\frac{(\lambda-\alpha)^{\lambda-\alpha}}{\lambda^\lambda} \cdot \inf_{x > 0} \frac{f_\lambda(x/\lambda)}{(x/\lambda)^\alpha}\right] \\
            &= \alpha^\alpha \cpc_\alpha(p) \cdot \lim_{\lambda \to \infty} \left[\frac{(\lambda-\alpha)^{\lambda-\alpha}}{\lambda^{\lambda-\alpha}} \cdot \cpc_\alpha(f_\lambda(x/\lambda))\right].
    \end{split}
    \]
    Notice that $\lim_{\lambda \to \infty} \cpc_\alpha(f_\lambda(x/\lambda)) = \cpc_\alpha(f)$ by Corollary \ref{capcont}. So we just need to compute the limit of the scaling factor:
    \[
        \lim_{\lambda \to \infty} \left(\frac{\lambda-\alpha}{\lambda}\right)^{\lambda-\alpha} = \lim_{\lambda \to \infty} \prod_{k=1}^n \left(1 - \frac{\alpha_k}{\lambda_k}\right)^{\lambda_k - \alpha_k} = \prod_{k=1}^n e^{-\alpha_k} = e^{-\alpha}.
    \]
    This completes the proof.
\end{proof}

\begin{corollary}
    Fix $f,g \in \mathcal{LP}_+[x_1,...,x_n]$. For any $\alpha \in \R_+^n$ we have:
    \[
        \langle f, g \rangle^\infty \geq \alpha^\alpha e^{-\alpha} \cpc_\alpha(f) \cpc_\alpha(g).
    \]
\end{corollary}
\begin{proof}
    Apply the previous theorem and Corollary \ref{capcont} to a sequence of real stable polynomials $g_k \to g$.
\end{proof}

\subsection{From Inner Products to Linear Operators}

The main purpose of this section, aside from proving the main technical result of the paper, is to demonstrate the power of a certain interpretation of the symbol of a linear operator. We will show that a simple observation regarding the symbol (which is explicated in more detail in \cite{leake2017representation}) will immediately enable us to transfer inner product bounds to bounds on linear operators. We now state this observation, which could be considered as a more algebraic definition of the symbol.

\begin{lemma}
    Let $\langle \cdot, \cdot \rangle$ be either $\langle \cdot, \cdot \rangle^\lambda$ or $\langle \cdot, \cdot \rangle^\infty$, and let $\Symb$ be either $\Symb^\lambda$ or $\Symb^\infty$, respectively. Let $T$ be a linear operator on polynomials of appropriate degree, and let $p,q \in \R_+[x_1,...,x_n]$ be polynomials of appropriate degree. Then we have the following, where the inner product acts on the $z$ variables:
    \[
        T[p](x) = \langle \Symb[T](z,x), p(z) \rangle.
    \]
\end{lemma}
\begin{proof}
    This is straightforward, as the scalars present in the expressions of $\langle \cdot, \cdot \rangle$ and $\Symb$ were chosen such that they cancel out in the above expression. We compute this on the monomial basis in the $\langle \cdot, \cdot \rangle^\lambda$ case, and the $\langle \cdot, \cdot \rangle^\infty$ case follows from a similar argument. For any $\mu \leq \lambda$ and $\nu$, let $T_{\mu \to \nu}$ be the linear operator on polynomials of degree $\lambda$ given by $T_{\mu \to \nu}[x^\mu] = x^\nu$ and $T_{\mu \to \nu}[x^\kappa] = 0$ for any other monomial $x^\kappa$. Recalling the bounded-degree definition of the symbol (Definition \ref{symbbounded}), we have:
    \[
        \Symb^\lambda(T_{\mu \to \nu}) = \binom{\lambda}{\mu} z^\mu x^\nu.
    \]
    For any monomial $z^\kappa$ with $\kappa \leq \lambda$, this implies the following for any fixed $x$:
    \[
    \begin{split}
        \langle \Symb^\lambda[T_{\mu \to \nu}](z,x), z^\kappa \rangle^\lambda &= \binom{\lambda}{\mu} x^\nu \cdot \langle z^\mu, z^\kappa \rangle^\lambda \\
            &= \binom{\lambda}{\mu} x^\nu \cdot \binom{\lambda}{\mu}^{-1} \delta_{\kappa=\mu} \\
            &= T_{\mu \to \nu}[x^\kappa].
    \end{split}
    \]
\end{proof}

As we will see very shortly, this will make for quick proofs of the main results given the inner product bounds we have already achieved. Before doing this though, let us discuss some of the linear operator bounds that Anari and Oveis~Gharan achieved in \cite{anari2017generalization}. Note the following differential operator form of $\langle \cdot, \cdot \rangle^\infty$:
\[
    \langle p, q \rangle^\infty = \left.q(\partial_x)q(x)\right|_{x=0}.
\]
Anari and Oveis~Gharan then use use their inner product bound to essentially give capacity preservation results for certain differential operators. Similarly, for multiaffine polynomials $\langle p, q \rangle^{(1^n)} = \left.q(\partial_x)q(x)\right|_{x=0}$, which gives a better bound in the multiaffine case. We now vastly generalize this idea, with a rather short proof.

\begin{theorem} \label{mainthmbounded}
    Let $T: \R_+^\lambda[x_1,...,x_n] \to \R_+^\gamma[x_1,...,x_m]$ be a linear operator such that $\Symb^\lambda(T)$ is real stable in $z$ for every $x \in \R_+^m$. Then for any real stable $p \in \R_+^\lambda[x_1,...,x_n]$, any $\alpha \in \R_+^n$, and any $\beta \in \R_+^m$ we have:
    \[
        \frac{\cpc_\beta(T(p))}{\cpc_\alpha(p)} \geq \frac{\alpha^\alpha (\lambda-\alpha)^{\lambda-\alpha}}{\lambda^\lambda} \cpc_{(\alpha,\beta)}(\Symb^\lambda(T)).
    \]
    Further, this bound is tight for fixed $T$, $\alpha$, and $\beta$.
\end{theorem}
\begin{proof}
    In the proof, let $\langle \cdot, \cdot \rangle := \langle \cdot, \cdot \rangle^\lambda$ and $\Symb := \Symb^\lambda$. By the previous lemma, we have the following for any fixed $x_0 \in \R_+^n$ (here, the inner product acts on the $z$ variables):
    \[
        T(p)(x_0) = \langle \Symb(T)(z,x_0), p(z) \rangle.
    \]
    Theorem \ref{boundbounded} then implies:
    \[
    \begin{split}
        T(p)(x_0) &= \langle \Symb(T)(z,x_0), p(z) \rangle \\
            &\geq \frac{\alpha^\alpha (\lambda-\alpha)^{\lambda-\alpha}}{\lambda^\lambda} \cpc_\alpha(p) \cdot \cpc_\alpha(\Symb(T)(\cdot,x_0)).
    \end{split}
    \]
    Dividing by $x_0^\beta$ on both sides and taking $\inf$ gives:
    \[
        \inf_{x_0 > 0} \frac{T(p)(x_0)}{x_0^\beta} \geq \frac{\alpha^\alpha (\lambda-\alpha)^{\lambda-\alpha}}{\lambda^\lambda} \cpc_\alpha(p) \cdot \inf_{x_0 > 0} \inf_{z > 0} \frac{\Symb(T)(z,x_0)}{z^\alpha x_0^\beta}.
    \]
    This is the desired result. Tightness then follows from considering input polynomials of the form $p(x) = \prod_k (1+x_ky_k)$ for fixed $y \in \R_+^n$, and then taking $\inf$ over $y$.
\end{proof}

As stated in the introduction, this is our main technical result, and we have already discussed some of its applications in \S\ref{applications}. We give a similar result for linear operators on polynomials of any degree.



\begin{theorem} \label{mainthmtrans}
    Let $T: \R_+[x_1,...,x_n] \to \R_+[x_1,...,x_m]$ be a linear operator such that $\Symb^\infty(T)$ is in $\mathcal{LP}_+[z_1,...,z_n]$ for every $x \in \R_+^m$. Then for any $p \in \R_+[x_1,...,x_n]$, any $\alpha \in \R_+^n$, and any $\beta \in \R_+^m$ we have:
    \[
        \frac{\cpc_\beta(T(p))}{\cpc_\alpha(p)} \geq e^{-\alpha} \alpha^\alpha \cpc_{(\alpha,\beta)}(\Symb^\infty(T)).
    \]
    Further, this bound is tight for fixed $T$, $\alpha$, and $\beta$.
\end{theorem}
\begin{proof}
    The proof given above for Theorem \ref{mainthmbounded} can be essentially copied verbatim.
\end{proof}

We now combine these results with the Borcea-Br\"and\'en characterization results (Theorems \ref{bbbounded} and \ref{bbtrans}) to give concrete corollaries which directly relate to stability preservers.

\begin{corollary} \label{maincorbounded}
    Suppose $T: \R_+^\lambda[x_1,...,x_n] \to \R_+^\gamma[x_1,...,x_m]$ is a linear operator of rank greater than 2, such that $T$ preserves real stability. Then Theorem \ref{mainthmbounded} applies to $T$.
\end{corollary}
\begin{proof}
    Since the image of $T$ is of dimension greater than 2, Theorem \ref{bbbounded} implies one of two possibilities:
    \begin{enumerate}
        \item $\Symb^\lambda[T]$ is real stable.
        \item $\Symb^\lambda[T](z_1,...,z_n,-x_1,...,-x_n)$ is real stable.
    \end{enumerate}
    In either case, we have that $\Symb^\lambda[T]$ is real stable in $z$ for every fixed $x \in \R_+^m$ (see Proposition \ref{rsclosure}). Therefore Theorem \ref{mainthmbounded} applies.
\end{proof}

\begin{corollary} \label{maincortrans}
    Suppose $T: \R_+[x_1,...,x_m] \to \R_+[x_1,...,x_m]$ is a linear operator of rank greater than 2, such that $T$ preserves real stability. Then Theorem \ref{mainthmtrans} applies to $T$.
\end{corollary}
\begin{proof}
    The same proof works, using Theorem \ref{bbtrans} instead.
\end{proof}

\section{Continuity of Capacity} \label{continuity}

In this section, we discuss the continuity of capacity as a function of the the input polynomial $p$. The main result of this section allows us to limit inner product bounds from $\langle \cdot, \cdot \rangle^\lambda$ to $\langle \cdot, \cdot \rangle^\infty$, which is exactly how we proved Theorem \ref{boundtrans}.

Given a (positive) discrete measure $\mu$ on $\R^n$, we define its generating function as:
\[
    p^\mu(x) := \sum_{\kappa \in \supp(\mu)} \mu(\kappa) x^\kappa.
\]
(Note that we have only restricted $\supp(\mu)$ to be in $\R^n$, and so $p_\mu$ may not be a polynomial.) We further define the log-generating function of $\mu$ as:
\[
    P^\mu(x) := \log(p^\mu(\exp(x))) = \log \sum_{\kappa \in \supp(\mu)} \mu(\kappa) \exp(x \cdot \kappa).
\]
More generally for such a function $p(x)$, we will write:
\[
    p(x) := \sum_\kappa p_\kappa x^\kappa,
\]
\[
    P(x) := \log(p(\exp(x))) = \log \sum_\kappa p_\kappa \exp(x \cdot \kappa).
\]
We care about discrete measures (with not necessarily finite support) whose generating functions are convergent and continuous on $\R_+^n$. This is equivalent to the log-generating function being continuous on $\R^n$. Note that an important example of such a measure is one which has finite support entirely in $\Z_+^n$. The generating functions of such measures are polynomials.

From now on we will write $\supp(p) = \supp(P)$ to denote the support of $\mu$ (as above) and $\Newt(p) = \Newt(P)$ to denote the polytope generated by its support. We first give a few basic results.

\newtheorem*{newtoncap}{Fact \ref{newtoncap}}
\begin{newtoncap}
    For $p$ a continuous generating function, the following are equivalent.
    \begin{enumerate}
        \item $\alpha \in \Newt(p)$.
        \item $\cpc_\alpha p(x) > 0$.
        \item $P(x) - \alpha \cdot x$ is bounded below.
    \end{enumerate}
\end{newtoncap}

\begin{lemma}
    Any continuous log-generating function $Q(x)$ is convex in $\R^n$.
\end{lemma}
\begin{proof}
    H\"older's inequality.
\end{proof}

Note that proving statements for $p$ is essentially the same as proving for $P$, as suggested in the following lemma.

\begin{lemma}
    Let $p,p_n$ be continuous generating functions. Then $p_n \to p$ uniformly on compact sets of $\R_+^n$ iff $P_n \to P$ uniformly on compact sets of $\R^n$.
\end{lemma}
\begin{proof}
    Equivalence of $p_n \to p$ and $\exp(P_n) \to \exp(P)$ follows form the fact that $\exp : \R_+^n \to \R^n$ is a homeomorphism (and so gives a bijection of compact sets). The fact that $\exp$ and $\log$ are (uniformly) continuous on every compact set in their domains then completes the proof.
\end{proof}

We now get the first half of the desired equality, which is the easier half.

\begin{lemma}
    With $p,p_n$ continuous generating functions and $p_n \to p$ uniformly on compact sets, we have:
    \[
        \lim_{n \to \infty} \inf p_n \leq \inf p.
    \]
\end{lemma}
\begin{proof}
    Let $(x_m) \subset \R_+^n$ be a sequence such that $p(x_m) \to \inf p$. For each $m$ we have that $p_n(x_m)$ is eventually near to $p(x_m)$. So for any fixed $\epsilon > 0$, we have the following for $m = m(\epsilon)$ and $n \geq N(\epsilon,m)$:
    \[
        \inf p_n \leq p_n(x_m) \leq p(x_m) + \epsilon \leq \inf p + 2\epsilon.
    \]
    The result follows by sending $\epsilon \to 0$.
\end{proof}

We now set out to prove the second half of the desired equality, the difficulty for which arises whenever $\alpha$ is on the boundary of $\Newt(p)$. 

\begin{lemma}
    Suppose $0$ is in the interior of $\Newt(p)$. Then $\inf P$ is attained precisely on some compact convex subset $K$ of $\R^n$.
\end{lemma}
\begin{proof}
    By a previous lemma, $\inf P$ is finite. Suppose $x_n$ is an unbounded sequence (with monotonically increasing norm) such that $P(x_n)$ limits to $\inf P$. By compactness of the $n$-dimensional sphere, we can assume by restricting to a subsequence that $\frac{x_n}{\|x_n\|}$ limits to some $u$. Pick $\epsilon > 0$ small enough such that $\epsilon u \in \Newt(p)$, and consider $P(x) - \epsilon u \cdot x$. We then have:
    \[
        \lim_{n \to \infty} P(x_n) - \epsilon u \cdot x_n = \lim_{n \to \infty} P(x_n) - \epsilon \|x_n\| \left(u \cdot \frac{x_n}{\|x_n\|}\right) = -\infty.
    \]
    However, since $\epsilon u \in \Newt(p)$ we have that $P(x) - \epsilon u \cdot x$ is bounded below, a contradiction. So, every sequence limiting to $\inf P$ is bounded, and therefore $\inf P$ is attained on a bounded set. By convexity of $P$, this set is convex.
\end{proof}

The next few results then finish the proof of continuity of $\cpc_\alpha(\cdot)$ under certain support conditions.

\begin{proposition}
    Let $p$ and $p_n$ be continuous generating functions such that $p_n \to p$, with $0$ in the interior of $\Newt(p)$. Then:
    \[
        \lim_{n \to \infty} \inf p_n = \inf p.
    \]
\end{proposition}
\begin{proof}
    Given the above lemma, we only have the $\geq$ direction left to prove. Since $0$ is in the interior of $\Newt(p)$, there is some compact convex $K \subset \R^n$ such that $P(x) = \inf P$ iff $x \in K$. Further, this implies that for any compact set $K'$ whose interior contains $K$, there exists $c_0 > 0$ such that $P(x) > \inf P + c_0$ on the boundary of $K'$. For any fixed positive $\epsilon < \frac{c_0}{2}$ and large enough $n$, we then have:
    \[
        |P_n - P| < \epsilon < \frac{c_0}{2} \text{ in } K' \Longrightarrow |P_n - \inf P| < \epsilon < \frac{c_0}{2} \text{ in } K,
    \]
    \[
        P_n > \inf P + (c_0 - \epsilon) > \inf P + \frac{c_0}{2} \text{ on the boundary of } K'.
    \]
    Convexity of $P_n$ then implies $P_n(x) > \inf P + \frac{c_0}{2}$ outside of $K'$. Therefore for any $\epsilon$ and large enough $n$:
    \[
        \inf P_n = \inf_{x \in K'} P_n \geq \inf P - \epsilon.
    \]
    Letting $\epsilon \to 0$ gives the result.
\end{proof}

We now set out to prove a similar statement whenever $0$ is on the boundary on $\Newt(p)$. This ends up needing a bit more restriction.

\begin{lemma}
    Suppose $0$ is on the boundary on $\Newt(P)$. Then there exists $A \in SO_n(\R)$ such that:
    \[
        \Newt(A \cdot P) \subset \{\kappa: \kappa_n \geq 0\},
    \]
    \[
        \inf \left.(A \cdot P)\right|_{x_n=-\infty} = \inf P.
    \]
\end{lemma}
\begin{proof}
    Since $0$ is on the boundary of the convex set $\Newt(P)$, a separating hyperplane gives a unit vector $c$ such that $(c|\mu) \geq 0$ for all $\mu \in \Newt(P)$. Let $A \in SO_n(\R)$ be such that $Ac = e_n$. We first have:
    \[
        \inf A \cdot P = \inf P(A^{-1}x) = \inf P.
    \]
    Since $\Newt(A \cdot P) = A \cdot \Newt(P)$ and $(e_n|A\mu) = (c|\mu) \geq 0$ for every $\mu \in \Newt(P)$, we have that $\Newt(A \cdot P) \subset \{\kappa : \kappa_n \geq 0\}$. Therefore:
    \[
        \inf \left.(A \cdot P)\right|_{x_n=-\infty} = \inf A \cdot P = \inf P.
    \]
    Note that $\left.(A \cdot P)\right|_{x_n=-\infty}$ denotes the continuous log-generating function given by the terms $\kappa$ of the support of $A \cdot P$ for which $\kappa_n = 0$. This is justified, as $\Newt(A \cdot P) \subset \{\kappa : \kappa_n \geq 0\}$ implies that $A \cdot P$ decreases as $x_n$ decreases (and we care about $\inf$).
\end{proof}

\begin{theorem}
    Let $p$ and $p_m$ be continuous generating functions such that $p_m \to p$, with $0 \in \Newt(p)$. Suppose further that eventually $\Newt(p_m) \subseteq \Newt(p)$. Then:
    \[
        \lim_{m \to \infty} \inf p_m = \inf p.
    \]
\end{theorem}
\begin{proof}
    Given the above proposition, we only need to prove this in the case where $0$ is on the boundary of $\Newt(p)$. In that case, the previous lemma gives an $A \in SO_n(\R)$ such that $\Newt(A \cdot P) \subset \{\kappa : \kappa_n \geq 0\}$ and $\inf \left.(A \cdot P)\right|_{x_n=-\infty} = \inf P$. Since $P_m \to P$ implies $A \cdot P_m \to A \cdot P$, we now relax to proving $\lim_{m \to \infty} \inf A \cdot P_m = \inf A \cdot P$. By assumption, eventually $\Newt(P_m) \subseteq \Newt(P)$ which implies $\Newt(A \cdot P_m) \subseteq \Newt(A \cdot P) \subset \{\kappa : \kappa_n \geq 0\}$. So, eventually $\Newt(\left.(A \cdot P_m)\right|_{x_n=-\infty}) \subseteq \Newt(\left.(A \cdot P)\right|_{x_n=-\infty})$ and $\inf A \cdot P_m = \inf \left.(A \cdot P_m)\right|_{x_n=-\infty}$. By induction on the number of variables, we then have:
    \[
        \lim_{m \to \infty} \inf A \cdot P_m = \lim_{m \to \infty} \inf \left.(A \cdot P_m)\right|_{x_n=-\infty} = \inf \left.(A \cdot P)\right|_{x_n=-\infty} = \inf A \cdot P.
    \]
    For the base case, $p_m$ and $p$ are scalars and the result is trivial.
\end{proof}

\begin{corollary} \label{capcont}
    Let $p_n$ be polynomials with nonnegative coefficients and $p$ analytic such that $p_n \to p$, with $\alpha \in \Newt(p)$. Then:
    \[
        \lim_{n \to \infty} \cpc_\alpha p_n = \cpc_\alpha p.
    \]
\end{corollary}
\begin{proof}
    As in the previous proposition, we only have the $\geq$ direction to prove. Let $q_n$ be defined as the sum of the terms of $p_n$ which appear in the support of $p$. Since the $p_n$ are polynomials with nonnegative coefficients, we have that $q_n \to p$. By the previous theorem, we then have:
    \[
        \lim_{n \to \infty} \cpc_\alpha p_n \geq \lim_{n \to \infty} \cpc_\alpha q_n = \cpc_\alpha p.
    \]
\end{proof}

Note that the fact that $q_n \to p$ holds after restricting to the support of $p$ relies on the fact that $p_n$ and $q_n$ are polynomials with positive coefficients. This is the main barrier to generalizing this corollary to all continuous generating functions.

\section{Concluding Remarks}

We have given here tight bounds on capacity preserving operators related to real stable polynomials. These results are essentially corollaries of inner product bounds, extended from bounds of Anari and Oveis~Gharan, all eventually based on the strong Rayleigh inequalities. That said, there are a number of pieces of this that may be able to be altered or generalized, and this raises new questions.

The first is that of the inner product: are there other inner products for which we can obtain bounds? The main conjecture in this direction is that of Gurvits in \cite{gurvits2009multivariate}.

\begin{conjecture}[Gurvits]
    Let $p,q \in \R_+[x_1,...,x_n]$ be homogeneous real stable polynomials of total degree $d$. Then:
    \[
        \sum_{\|\mu\|_1 = d} \binom{d}{\mu}^{-1} p_\mu q_\mu \geq \frac{\alpha^\alpha}{d^d} \cpc_\alpha(p) \cpc_\alpha(q).
    \]
\end{conjecture}

The main difference here is that we use multinomial coefficients rather than products of binomial coefficients. Note that the symbol operator associated to this inner product is given by $T[(z \cdot x)^d]$ (dot product of $z$ and $x$). It is not immediately clear how this inner product relates to real stable polynomials, as the link to stability preservers is less clear than in the Borcea-Br\"and\'en case.

The next is the class of polynomials: are there more general classes of polynomials for which weaker capacity bounds can be achieved? One such bound is achieved for \emph{strongly log-concave polynomials} (originally studied by Gurvits in \cite{gurvits2009multivariate}) in \cite{anari2018log}, and this class contains basis generating polynomials of matroids. (Note that the authors call these polynomials \emph{completely log-concave}, and they are also called \emph{Lorentzian} in \cite{branden2019lorentzian}.) This bound relies on a weakened version of the strong Rayleigh inequalities, where a factor of 2 is introduced. It is unclear what applications such a bound has beyond those of \cite{anari2018log}.

The last is a question about the further applicability of the main results of this paper. In particular, all of the operators studied here are differential operators. Are there applications of non-differential operators? Also, are there ways to get a handle on the location of the roots of a polynomial via capacity? This second question is of particular interest, as it may lead to a more unified and a direct approach to various root bounding results. For example, the root bounds of \cite{finiteconvolutions} are at the heart of the proof of the Kadison-Singer conjecture in \cite{marcus2015interlacing}. Can capacity be used to achieve those bounds?

\subsection*{Acknowledgements}

The second author is thankful to Peter Csikv\'ari for discussions on his bound and on a number of related topics and results, and also to Nick Ryder for discussions on topics related to the polynomial theory used in this paper. The authors are also grateful to Igor Pak for very helpful comments.

\bibliographystyle{amsalpha_titles}
\bibliography{bibliography}

\end{document}